\documentclass[a4paper,12pt,reqno]{amsart}
\usepackage{amsfonts, color}
\usepackage{amsmath}
\usepackage{amssymb}
\usepackage[a4paper]{geometry}
\usepackage{mathrsfs}
\usepackage[colorlinks]{hyperref}
\usepackage{bigints}
\usepackage{hyperref}
\renewcommand\eqref[1]{(\ref{#1})}

\numberwithin{equation}{section}
\theoremstyle{plain}
\newtheorem{theorem}{Theorem}[section]
\newtheorem{corollary}{Corollary}[section]

\theoremstyle{definition}
\newtheorem{remark}{Remark}
\newtheorem{example}{Example}

\newcommand{\Rn}{\mathbb{R}^n}

\usepackage[mathlines]{lineno}

\usepackage[english]{babel}
\usepackage{graphicx}
\usepackage{textcomp}
\usepackage{amssymb}
\usepackage{amsmath}
\usepackage{setspace}
\usepackage{geometry}
\usepackage{bm}
\usepackage{centernot}
\usepackage{hyperref}
\usepackage{amsthm}

\newtheorem*{assumptionsK}{Assumptions $K$}
\newtheorem*{assumptionsg}{Assumptions $g$}
\newtheorem*{assumptionsh}{Assumptions $h$}
\newtheorem*{assumptionsG}{Assumptions $G$}
\newtheorem*{assumptionsSt}{Assumptions $\lambda_\pm=0$}
\newtheorem*{assumptionssubSt}{Assumptions $\lambda_\pm>0$}

\newtheorem*{assumptionsgAlt}{Assumptions $\tilde g$}
\newtheorem*{assumptionsGAlt}{Assumptions $\tilde G$}
\newtheorem*{assumptionsStAlt}{Assumptions $\widetilde{\lambda_\pm=0}$}
\newtheorem*{assumptionssubStAlt}{Assumptions $\widetilde{\lambda_\pm>0}$}

\usepackage{xcolor}

\renewcommand{\ae}{{\mathrm{a.e.}\;}}

\title{Parabolic equations with concave non-linearity}

\author[Z. Avetisyan]{Zhirayr Avetisyan}
\address{
	Zhirayr Avetisyan:
	\endgraf
    Department of Mathematics: Analysis, Logic and Discrete Mathematics
    \endgraf
    Ghent University, Belgium
  	\endgraf
	{\it E-mail address} {\rm zhirayr.avetisyan@ugent.be}
		}
	
\author[Kh. A. Khachatryan]{Khachatur Khachatryan}
\address{
  Khachatur Khachatryan:
  \endgraf
   \endgraf
  Chair of Theory of Functions  and Differential Equations
    \endgraf
  Faculty of Mathematics, Yerevan State University, Armenia
  \endgraf
	{\it E-mail address} {\rm khachatur.khachatryan@ysu.am}}

\author[M. Ruzhansky]{Michael Ruzhansky}
\address{
  Michael Ruzhansky:
  \endgraf
  Department of Mathematics: Analysis, Logic and Discrete Mathematics
  \endgraf
  Ghent University, Belgium
  \endgraf
 and
  \endgraf
  School of Mathematical Sciences
  \endgraf
  Queen Mary University of London
  \endgraf
  United Kingdom
  \endgraf
  {\it E-mail address} {\rm michael.ruzhansky@ugent.be}
  }

\begin{document}

\thanks{The first and third authors are supported by the FWO Odysseus 1 grant G.0H94.18N: Analysis and Partial Differential Equations and by the Methusalem programme of the Ghent University Special Research Fund (BOF) (Grant number 01M01021). Michael Ruzhansky is also supported by EPSRC grant EP/V005529/1 and FWO Senior
Research Grants G011522N and G022821N. The research of the second author was supported by the Science Committee of RA (Research project no. 23RL-1A027) 
}

\begin{abstract}
In this paper we prove the existence and uniqueness of positive mild solutions for the
semilinear parabolic equations of the form $u_t+\mathcal{L}u=f+h\cdot G(u)$, where $h$ is a positive function and $G$ a positive concave function (for example, $G(u)=u^\alpha$ for $0<\alpha<1$). In contrast with the case of convex $G$, where the Fujita exponent appears, and only existence of a positive solution for special data is achieved, in this concave case we obtain the existence and uniqueness for any positive data.
The method relies on proving the existence and uniqueness of solutions for a certain class of Hammerstein-Volterra-type integral equations with a concave non-linear term, in the very general setting of arbitrary measure spaces. As a consequence, we obtain results for the existence and uniqueness of mild solutions to semilinear parabolic equations with concave nonlinearities under rather general assumptions on $f, h$ and $G$, in a variety of settings: e.g. for $\mathcal{L}$ being the Laplacian on complete Riemannian manifolds (e.g., symmetric spaces of any rank) or H\"ormander sum of squares on general Lie groups. We also study the corresponding equations in the presence of a damping term, in which case assumptions can be relaxed even further.
\end{abstract}

\maketitle

\tableofcontents

\section{Introduction}

In this paper, we are interested in the following parabolic Cauchy problem,
\begin{equation}
\begin{cases}
\partial_t u(x,t)+\mathcal{L}u(x,t)=h(x,t)G(u(x,t))+f(x,t),\\
u(x,0)=u_0(x),
\end{cases}\label{CauchyProblem0}
\end{equation}
for an unknown non-negative function $u:X\times\mathbb{R}_+\to[0,+\infty)$, where $\mathbb{R}_+=(0,\infty)$, given non-negative initial data $u_0:X\to[0,+\infty)$.

Duhamel's principle allows one to reduce a semilinear parabolic equation to an integral equation, which can often be studied with different methods. Solutions to this integral equation are often referred to as mild solutions of the original parabolic equation. In this work, we study the existence and uniqueness of solutions to a certain class of Hammerstein-Volterra-type integral equations with a concave non-linear term, which will imply existence and uniqueness of non-negative mild solutions to \eqref{CauchyProblem0}, where $h$ is a positive function and $G$ a positive function satisfying certain concavity conditions.

More specifically, let $(X,\mu)$ be a measure space and let $\mathcal{L}:\mathcal{D}\to L^0(X)$ be a linear operator with domain $\mathcal{D}\subset L^0(X)$, $L^0(X)\doteq L^0(X,\mu)$ being the space of almost everywhere finite measurable functions. Let further $h:X\times\mathbb{R}_+\to\mathbb{R}_+$, $G:[0,+\infty)\to[0,+\infty)$ and $f:X\times\mathbb{R}_+\to\mathbb{R}$ be given functions. Thus, we study the so-called Hammerstein-Volterra-type integral equations of the form
\begin{equation}
u(x,t)=g(x,t)+\int\limits_0^t\int\limits_XK(x,y;t-s)h(y,s)G(u(y,s))d\mu(y)ds,\label{IntEq0}
\end{equation}
where $g$ is the function corresponding to the initial data $u_0$ in \eqref{CauchyProblem0}, and $K$ is the heat kernel for the operator $\mathcal{L}$ (see below for details).
The only assumptions that we will need on $K:X\times X\times\mathbb{R}_+\to [0,\infty)$  is that it is measurable and that there exist
$\lambda_-,\lambda_+$, with $\lambda_+\ge\lambda_-\ge0$, such that
\begin{equation}
e^{-\lambda_+t}\le\int\limits_XK(x,y;t)d\mu(y)\le e^{-\lambda_-t},\quad\ae x\in X,\quad\ae t\in\mathbb{R}_+.\label{KAssumptions_ii0}
\end{equation}
The case $\lambda_-=\lambda_+=0$ will be referred to as the stochastic case and corresponds to the situation when the heat kernel is a probability density. The case of
$\lambda_+\ge\lambda_->0$ in \eqref{KAssumptions_ii0} will be referred to as the substochastic case, and here our assumptions on $G$ and $h$ can be relaxed compared to the stochastic case. Thus, from the point of view of the operator $\mathcal{L}$, we cover the following situations:
\begin{itemize}
\item[(i)] $\mathcal{L}$ is the Laplace-Beltrami operator on a complete Riemannian manifold with Ricci curvature bounded from below. Or, for complete connected Riemannian manifolds without restrictions on the curvature, with volume growth bounded by $\exp(c\,r^2)$. These cases cover, for example, Riemannian symmetric spaces of non-compact type of general rank.
\item[(ii)] $\mathcal{L}$ is the H\"ormander sum of squares (sub-Laplacian) on a general Lie group.
\item[(iii)] $\mathcal{L}$ is the sub-Laplacian on sub-Rieman\-nian manifolds satisfying the so-called curvature-dimension inequalities. This includes sub-Riemannian manifolds with transverse symmetries.
\item[(iv)] $\mathcal L$ is a general densely defined linear operator on $L^2(X)$ for a general measure space $(X,\mu)$, such that the semigroup of measures corresponding to $e^{-t\mathcal{L}}$ is a family of uniformly bounded positive measures (e.g. probability measures, which is the stochastic case) which are absolutely continuous with respect to $\mu$ for a.e. $t$.
\end{itemize}
We refer to Section \ref{SEC:K} for a more detailed discussion of these settings. The substochastic case of
$\lambda_+\ge\lambda_->0$ in \eqref{KAssumptions_ii0} requires that the operator $\mathcal{L}$ has a spectral gap. This condition is necessary but not sufficient for substochasticity; for instance, the Laplacian on symmetric spaces of non-compact type has a spectral gap, but is still stochastic. The substochastic case is necessary in order to cover, for instance, the semilinear parabolic equation with a damping term. Namely, when we consider the equation with the stochastic assumption on $\mathcal{L}$ but we add a mass term to it,
\begin{equation}
\begin{cases}
\partial_t u(x,t)+\mathcal{L}u(x,t)+m u(x,t)=h(x,t)G(u(x,t))+f(x,t),\\
u(x,0)=u_0(x),
\end{cases}\label{CauchyProblem02}
\end{equation}
with a positive mass term $m>0$ (for $\mathcal L$ also positive). Denoting $\mathcal{L}'\doteq\mathcal{L}+m$ we obtain the original Cauchy problem (\ref{CauchyProblem0}), but now $\mathcal{L}'$ is substochastic with $\lambda_+=\lambda_-=m$.

Equations of the kind \eqref{CauchyProblem0} or \eqref{CauchyProblem02} with convex $G$ have been studied widely in the literature. The famous result in 1966 by Fujita \cite{fuj1} states that the Cauchy problem for the equation
\begin{equation}\label{Heat_Cauchyrn}
\begin{split}
& u_{t}(t,x)-\Delta u(t,x) =u^p(t,x),\,\, x\in \mathbb{R}^N,\, t>0, \\&
u(0,x)=u_0(x)\geq 0,\, x\in \mathbb{R}^N,
   \end{split}
\end{equation}
where $u_0\not\equiv 0$, has no global in time solution for $p$ in the range $1<p<p_F=1+\frac{2}{N}$, where  $p_F$ is the Fujita exponent. It was also shown in \cite{hayak2} and \cite{AW78} that this result holds also for $p = p_F$. Fujita also proved that for $p>p_F$ there exists a positive solution for some positive data. Moreover, it was shown by Weissler \cite{Wei81} that if $p>p_{F}$, and $u_{0}\in L^{\infty}(\Rn)$ is non-negative and  small, namely, satisfying
$$\int_{0}^{\infty}\|\text{e}^{-s\Delta}u_{0}\|_{L^{\infty}(\Rn)}^{p-1}ds<\frac{1}{p-1},$$
then there exists a non-negative global solution of \eqref{Heat_Cauchyrn}.
Thus, we see that in the case of convex nonlinearity like $u^p$ for $p>1$, we have the existence of a threshold between global existence and blow-up of solutions in terms of the order $p$.
An exposition of this type of results in the Euclidean settings can be found, e.g., in the books \cite{Soup} or \cite{Lev90}.
Equations of the kind \eqref{IntEq0} for convex nonlinearities have been also studied, as they often appear in applications (e.g. \cite{Diek} with applications to the analysis of the spread of infectious deceases).

There are many extensions of these results to other settings such as bounded and unbounded domains, Riemannian symmetric spaces, sub-Riemannian manifolds, or Lie groups.
For example, the Fujita exponent for general H\"ormander sums of squares on general unimodular Lie groups was obtained in \cite{RY}. There, the case of sub-Riemannian manifolds was also studied, and the analysis depends on estimates for the heat kernels. For the particular case of the Heisenberg groups  the problem was also intensively studied in the literature, see e.g. \cite{Kirane 1, Kirane 2, BRT22, Birindelli, D'Ambrosio, D'Ambrosio1, CKR24}.  In  \cite{CKR24}, the problem \eqref{CauchyProblem0} was considered on general unimodular Lie groups with
$h(x,t)=h(t)$ depending only on $t$ and one also obtained
the necessary and sufficient conditions for the global well-posedeness of the corresponding heat equation on the Heisenberg group $\mathbb{H}^1$, as well as sufficient conditions for the existence of positive solutions for small data on general unimodular Lie groups.

In all these cases of convex nonlinearities $G(u)$ in \eqref{CauchyProblem0}, one has the appearance of the threshold for the existence of global solutions for small data, and methods often depend on having suitable estimates for the heat kernels.
As we will show in this paper, for concave nonlinearities like $G(u)=u^\alpha$ for $0<\alpha<1$, there are no such thresholds, and also no smallness on the data is required. Moreover, no estimates on the heat kernel are needed, except for its non-negativity and stochasticity or substochasticity.

Conditions that we will impose on the nonlinearity $G(u)$ and function $h(x,t)$ in \eqref{CauchyProblem0} will be formulated precisely in the next sections. However, let us already give several examples, to illustrate the applicability of the obtained results. Thus, we can consider the following examples, where we also give the corresponding value of the function $\varphi$, the existence of which is part of the assumptions:

\begin{itemize}
    \item [G1.]  $G(u)=u^{\alpha}, \quad u\in [0,+\infty), \quad \alpha \in(0,1), \quad \varphi(\sigma)=\sigma^{\alpha},$
    \item [G2.]  $G(u)=\gamma (1-e^{-u^{\alpha}}), \quad u\in [0, +\infty), \quad \alpha \in(0,1),\quad \gamma>1, \quad \varphi(\sigma)=\sigma^{\alpha},$
    \item [$\Lambda$1.] $h(x,t)=p_1(t)(1-\lambda_0(x))+p_2(t)\lambda_0(x), \quad \forall x\in X, \quad \forall t\in \mathbb{R}_+,$ where $0\leq \lambda_0(x)\leq 1, \quad \forall x\in X.$
\end{itemize}

Note that the examples G1 and G2 above satisfy all assumptions regarding the function $G$, namely, {\normalfont\textbf{Assumptions}} \hyperref[GAssumptionsGeneric]{$G$}, the Assumptions (\ref{ExtraAssumptionsl=0_iv}), (\ref{ExtraAssumptionsl=0_v}), (\ref{ExtraAssumptionsl=0_vi}), and the Assumption (\ref{ExtraAssumptionsl>0_ii}) for any choice of $\lambda_->0$, $\beta>0$ and $\gamma>0$. The example $\Lambda$1 satisfies {\normalfont\textbf{Assumptions}} \hyperref[l*AssumptionsGeneric]{$h$} and the Assumptions (\ref{ExtraAssumptionsl=0_i}), (\ref{ExtraAssumptionsl=0_ii}) and (\ref{ExtraAssumptionsl=0_iii}) for the right choice of the functions $p_1$,$p_2$. The satisfaction of the Assumption (\ref{ExtraAssumptionsl>0_i}) by the example $\Lambda$1 is an easy matter of choosing $p_1$, $p_2$ and $\lambda_0$.
Precise conditions on more general $G$ and $h$ will be formulated in the next sections.

\section{The setting}

Let $(X,\mu)$ be a measure space and let $\mathcal{L}:\mathcal{D}\to L^0(X)$ be a linear operator with domain $\mathcal{D}\subset L^0(X)$, $L^0(X)\doteq L^0(X,\mu)$ being the space of almost everywhere finite measurable functions. Let further $h:X\times\mathbb{R}_+\to\mathbb{R}_+$, $G:[0,+\infty)\to[0,+\infty)$ and $f:X\times\mathbb{R}_+\to\mathbb{R}$ be given functions, where $\mathbb{R}_+=(0,\infty)$. Formally, we are interested in the following parabolic Cauchy problem,
\begin{equation}
\begin{cases}
\partial_t u(x,t)+\mathcal{L}u(x,t)=h(x,t)G(u(x,t))+f(x,t),\\
u(x,0)=u_0(x),
\end{cases}\label{CauchyProblem}
\end{equation}
for an unknown non-negative function $u:X\times\mathbb{R}_+\to[0,+\infty)$, given non-negative initial data $u_0:X\to[0,+\infty)$. Let us introduce the auxiliary function $g:X\times\mathbb{R}_+\to\mathbb{R}$ as a formal solution of the auxiliary Cauchy problem
\begin{equation}
\begin{cases}
\partial_t g(x,t)+\mathcal{L}g(x,t)=f(x,t),\\
g(x,0)=u_0(x).
\end{cases}\label{AuxCauchyProblem}
\end{equation}
This problem can be solved by Duhamel's principle. Note that if $u_0\ge0$ and $f\ge0$ then $g\ge0$, provided that the heat kernel of $\mathcal{L}$ exists and is non-negative. Then it is not difficult to check that every formal solution of the integral equation
\begin{equation}
u(x,t)=g(x,t)+\int\limits_0^te^{-(t-s)\mathcal{L}}\left[h(\cdot,s)G(u(\cdot,s))\right](x)ds\label{IntEqFormal}
\end{equation}
is also a formal solution of the Cauchy problem (\ref{CauchyProblem}). Following the tradition (e.g., \cite{Bro64}), we will call solutions of (\ref{IntEqFormal}) \textit{mild} solutions of the Cauchy problem (\ref{CauchyProblem}). Normally, additional work and assumptions are needed to rigorously produce strong solutions from mild solutions. However, mild solutions are of their own interest, and in this paper we will concentrate on the solutions of the integral formulation \eqref{IntEqFormal}.

If we assume, still formally, that the heat semigroup $\{e^{-t\mathcal{L}}\}_{t\in\mathbb{R}_+}$ admits a locally integrable kernel $K:X\times X\times\mathbb{R}_+\to\mathbb{R}$ such that, at least formally,
$$
e^{-t\mathcal{L}}h(x)=\int\limits_XK(x,y;t)h(y)d\mu(y),
$$
then the equation (\ref{IntEqFormal}) can be written as a proper integral equation,
\begin{equation}
u(x,t)=g(x,t)+\int\limits_0^t\int\limits_XK(x,y;t-s)h(y,s)G(u(y,s))d\mu(y)ds.\label{IntEq}
\end{equation}
This Hammerstein-Volterra-type integral equation, together with its solutions, will be the central subject of the present work.

Below we will obtain existence and uniqueness results for the solutions of the equation (\ref{IntEq}). In order to do so, we will need to impose assumptions on mathematical objects appearing in that equation. For the integral kernel $K$ we will propose the following assumptions.

\begin{assumptionsK}[Heat kernel $K$]\label{KAssumptions} We assume that
\begin{itemize}

\item[(i)]
\begin{equation}
K:X\times X\times\mathbb{R}_+\to[0,+\infty)\quad\mbox{ is measurable};\label{KAssumptions_i}
\end{equation}

\item[(ii)] $\exists\lambda_-,\lambda_+$, with $\lambda_+\ge\lambda_-\ge0$, such that
\begin{equation}
e^{-t\lambda_+}\le\int\limits_XK(x,y;t)d\mu(y)\le e^{-t\lambda_-},\quad\ae x\in X,\quad\ae t\in\mathbb{R}_+.\label{KAssumptions_ii}
\end{equation}

\end{itemize}
\end{assumptionsK}
The case $\lambda_-=\lambda_+=0$ will be referred to as the stochastic case, and $\lambda_+\ge\lambda_->0$ as the substochastic case. These two cases will be dealt with in slightly different ways. More precisely, $\lambda_->0$ provides additional convenience, which allows to ease the assumptions for solvability.

For the inhomogeneous contribution $g$ in the equation (\ref{IntEq}) we will assume some or all of the following conditions.

\begin{assumptionsg}[Term $g$ in (\ref{IntEq})]\label{gAssumptions} We assume that
\begin{itemize}

\item[(i)]
\begin{equation}
g:X\times\mathbb{R}_+\to[0,+\infty)\quad\mbox{is measurable};\label{gAssumptions_i}
\end{equation}

\item[(ii)] $\exists\,T_0>0$ such that
\begin{equation}
\inf_{(x,t)\in X\times (0,T_0)}g(x,t)\doteq\beta_0>0;\label{gAssumptions_ii}
\end{equation}

\item[(iii)]
\begin{equation}
g\in L^\infty(X\times\mathbb{R}_+).\label{gAssumptions_iii}
\end{equation}

\end{itemize}
\end{assumptionsg}

For the function $h$ we adopt the following generic assumptions.

\begin{assumptionsh}[Term $h$ in (\ref{IntEq})]\label{l*AssumptionsGeneric} We assume that
\begin{itemize}

\item[(i)]
\begin{equation}
h:X\times\mathbb{R}_+\to[0,+\infty)\quad\mbox{is measurable};\label{l*AssumptionsGeneric_i}
\end{equation}

\end{itemize}
\end{assumptionsh}
We will impose additional conditions on the function $h$ depending on whether $\lambda_\pm=0$ or not in {\normalfont\textbf{Assumptions}} \hyperref[KAssumptions]{$K$}.

Regarding the function $G$ in the equation (\ref{IntEq}) we assume the following generic conditions.

\begin{assumptionsG}[Non-linearity $G$]\label{GAssumptionsGeneric} We assume that
\begin{itemize}

\item[(i)]
\begin{equation}
G\in C([0,+\infty),[0,+\infty));\label{GAssumptionsGeneric_i}
\end{equation}

\item[(ii)]
\begin{equation}
G(0)=0,\quad G\quad\mbox{is strictly increasing and concave}.\label{GAssumptionsGeneric_ii}
\end{equation}

\end{itemize}
\end{assumptionsG}
Again, depending on whether $\lambda_\pm=0$ or not in {\normalfont\textbf{Assumptions}} \hyperref[KAssumptions]{$K$}, there will be additional assumptions on the function $G$ below.

\section{Existence of solutions}
\subsection{The stochastic case: $\lambda_\pm=0$}

Suppose that $\lambda_\pm=0$ in {\normalfont\textbf{Assumptions}} \hyperref[KAssumptions]{$K$}, which we called the stochastic case. Here we will demand the following additional assumptions on the functions $h$ and $G$.

\begin{assumptionsSt}[Additional]\label{ExtraAssumptionsl=0} We assume that $\exists\,p_1,p_2\in C(\mathbb{R}_+,\mathbb{R}_+)$ such that
\begin{itemize}

\item[(i)]
\begin{equation}
p_1\not\equiv p_2,\quad\int\limits_0^\infty p_2(t)dt=1;\label{ExtraAssumptionsl=0_i}
\end{equation}

\item[(ii)]
\begin{equation}
\lim_{t\to0+}\,\frac{p_1(t)}{p_2(t)}>0;\label{ExtraAssumptionsl=0_ii}
\end{equation}

\item[(iii)]
\begin{equation}
p_1(t)\le h(x,t)\le p_2(t),\quad\ae x\in X,\quad\ae t\in\mathbb{R}_+.\label{ExtraAssumptionsl=0_iii}
\end{equation}

\end{itemize}
Suppose $\exists\,\varphi:[0,1]\to[0,1]$ homeomorphism such that
\begin{itemize}

\item[(iv)]
\begin{equation}
\varphi(0)=0,\quad\varphi(1)=1,\quad\varphi\quad\mbox{is concave};\label{ExtraAssumptionsl=0_iv}
\end{equation}

\item[(v)] $\exists\,\xi>0$ such that
\begin{equation}
\xi-G(\xi)=\|g\|_{L^\infty(X\times\mathbb{R}_+)};\label{ExtraAssumptionsl=0_v}
\end{equation}

\item[(vi)]
\begin{equation}
G(\sigma u)\ge\varphi(\sigma)G(u),\quad\forall\sigma\in[0,1],\quad\forall u\in[0,\xi].\label{ExtraAssumptionsl=0_vi}
\end{equation}

\end{itemize}
\end{assumptionsSt}

\begin{remark} Note that if the function $G$ is so-called strongly concave, i.e.,
$$
\lim_{u\to+\infty}\frac{G(u)}{u}=0,
$$
then the Assumption (\ref{ExtraAssumptionsl=0_v}) is automatically satisfied.
\end{remark}

The first existence result is the following

\begin{theorem}\label{Existencel=0Theorem} Suppose that the following assumptions hold:
\begin{itemize}

\item {\normalfont\textbf{Assumptions}} \hyperref[KAssumptions]{$K$} hold with $\lambda_\pm=0$;

\item {\normalfont\textbf{Assumptions}} $g$: \hyperref[gAssumptions_i]{(i)} and \hyperref[gAssumptions_iii]{(iii)};

\item {\normalfont\textbf{Assumptions}} \hyperref[l*AssumptionsGeneric]{$h$};

\item {\normalfont\textbf{Assumptions}} \hyperref[GAssumptionsGeneric]{$G$};

\item {\normalfont\textbf{Assumptions}} $\lambda_\pm=0$: \hyperref[ExtraAssumptionsl=0_i]{(i)}, \hyperref[ExtraAssumptionsl=0_iii]{(iii)} and \hyperref[ExtraAssumptionsl=0_v]{(v)}.

\end{itemize}
Then
$$
\exists\,u\in L^\infty(X\times\mathbb{R}_+,[0,+\infty)),
$$
such that $u$ satisfies the equation (\ref{IntEq}). Moreover, $u$ is the almost everywhere limit of the decreasing sequence $\{u_m\}_{m=1}^\infty$ of non-negative measurable functions, given by following iterative approximations,
$$
u_{m+1}(x,t)\doteq g(x,t)+\int\limits_0^t\int\limits_XK(x,y;t-s)h(y,s)G(u_m(y,s))d\mu(y)ds,
$$
$$
u_0(x,t)\doteq\xi-\beta+g(x,t),\quad\forall m\in\{0\}\cup\mathbb{N},\quad\ae x\in X,\quad\ae t\in\mathbb{R}_+,
$$
with $\beta\doteq\|g\|_{L^\infty(X\times\mathbb{R}_+)}$. If, in addition, the following assumptions hold:
\begin{itemize}

\item {\normalfont\textbf{Assumptions}} $g$: \hyperref[gAssumptions_ii]{(ii)};

\item {\normalfont\textbf{Assumptions}} $\lambda_\pm=0$: \hyperref[ExtraAssumptionsl=0_ii]{(ii)}, \hyperref[ExtraAssumptionsl=0_iv]{(iv)} and \hyperref[ExtraAssumptionsl=0_vi]{(vi)},

\end{itemize}
then there exist $C>0$ and $k\in(0,1)$ such that
$$
0\le u_{m+1}(x,t)-u(x,t)\le C\,k^m,\quad\ae x\in X,\quad\ae t\in\mathbb{R}_+,\quad\forall m\in\mathbb{N}.
$$
\end{theorem}
\begin{proof}
Denote by
\begin{equation}
v_m(x,t)\doteq u_m(x,t)-g(x,t), \quad\forall m\in\{0\}\cup\mathbb{N}, \quad \ae x\in X,\quad \ae t\in\mathbb{R}_+.\label{Eq4}
\end{equation}
Then the successive approximations take the following form,
\begin{equation}
\begin{array}{c}
\displaystyle v_{m+1}(x,t)=\int\limits_0^{t}\int\limits_{X}K(x,y;t-s)h(y,s) G(v_m(y,s)+g(y,s))d\mu(y)ds, \\
 \displaystyle  v_0(x,t)=\xi-\beta , \quad \forall m\in\{0\}\cup\mathbb{N}, \quad \ae x\in X, \quad \ae t\in\mathbb{R}_+.
\label{Eq5}
\end{array}
\end{equation}
By induction in $m$ it is easy to check that the sequence $\{v_m(x,t)\}_{m=0}^{\infty}$ consists of functions measurable on $X \times \mathbb{R}_+$ and
\begin{equation}
v_m(x,t)\geq 0, \quad\forall m\in\{0\}\cup\mathbb{N}, \quad \ae x\in X, \quad \ae t\in\mathbb{R}_+.\label{Eq6}
\end{equation}
Now let us show that
\begin{equation}
v_m(x,t)\,\,  \text{is monotonically decreasing in} \,\,m, \quad \ae x\in X, \quad \ae t\in\mathbb{R}_+.\label{Eq7}
\end{equation}
First let us establish that
$$
v_1(x,t)\leq v_0(x,t),\quad\ae x\in X,\quad\ae t\in\mathbb{R}_+.
$$
Indeed, in view of the monotonicity of $G$, together with the Assumption (\ref{KAssumptions_ii}) with $\lambda_\pm=0$, as well as the Assumptions (\ref{gAssumptions_iii}), (\ref{ExtraAssumptionsl=0_i}), (\ref{ExtraAssumptionsl=0_iii}) and (\ref{ExtraAssumptionsl=0_v}), we have
$$
v_1(x,t)\overset{(\ref{ExtraAssumptionsl=0_v}),(\ref{GAssumptionsGeneric_ii})}{\leq}\int\limits_0^t\int\limits_XK(x,y,t-s)h(y,s)G(\xi)d\mu(y)ds\overset{(\ref{ExtraAssumptionsl=0_iii})}{\leq}G(\xi)\int\limits_0^tp_2(s)ds
$$
$$
\overset{(\ref{ExtraAssumptionsl=0_i})}{\leq}G(\xi)=\xi-\beta=v_0(x,t),\quad\ae x\in X,\quad\ae t\in\mathbb{R}_+.
$$
Then, assuming that $v_m(x,t)\le v_{m-1}(x,t)$ for $\ae x\in X$, $\ae t\in\mathbb{R}_+$ and some $m\in\mathbb{N}$, one can show that $v_{m+1}(x,t)\le v_m(x,t)$ for $\ae x\in X$, $\ae t\in\mathbb{R}_+$, and apply the principle of mathematical induction to arrive at (\ref{Eq7}). Let $v\in L^\infty(X\times\mathbb{R}_+)$ be the almost everywhere limit (or infimum) of the monotonically decreasing non-negative sequence $\{v_m\}_{m=1}^\infty$. Since $G$ possesses the properties in {\normalfont\textbf{Assumptions}} \hyperref[GAssumptionsGeneric]{$G$}, taking into account {\normalfont\textbf{Assumptions}} \hyperref[KAssumptions]{$K$} and Assumption (\ref{gAssumptions_iii}), according to B. Levi's theorem \cite{kol} we conclude that the limit function $v$ satisfies the equation
\begin{equation}
v(x,t)=\int\limits_0^t \int\limits_X K(x,y;t-s)h(y,s)G(v(y,s)+g(y,s))d\mu(y)ds,\;\; \ae x\in X, \;\;  \ae t\in \mathbb{R}_+. \label{Eq34}
\end{equation}
Due to (\ref{Eq4}), from (\ref{Eq34}) we conclude that $u\doteq v+g$ is a desired solution of the equation (\ref{IntEq}). Due to (\ref{Eq6}) and (\ref{Eq7}), we also obtain that
\begin{equation}
0\leq v(x,t)\leq \xi-\beta, \quad \ae x\in X, \quad  \ae t\in \mathbb{R}_+. \label{Eq32}
\end{equation}

{\bf With the additional assumptions}, let us next prove that (with $\beta_0$ defined in (\ref{gAssumptions_ii}))
\begin{equation}
\text{for}\quad\ae  (x,t)\in X\times (0,T_0), \quad \text{we have} \quad v_1(x,t)\geq G(\xi-\beta+\beta_0)\int\limits_0^t p_1(s)ds, \label{Eq8}
\end{equation}
\begin{equation}
\text{and for}\quad\ae  (x,t)\in X\times [T_0, +\infty), \quad \text{we have} \quad v_1(x,t)\geq G(\xi-\beta)\int\limits_0^t p_1(s)ds. \label{Eq9}
\end{equation}
Indeed, taking into account {\normalfont\textbf{Assumptions}} \hyperref[gAssumptions]{$g$}, the monotonicity of the function $G$, as well as the inequality in the Assumption (\ref{ExtraAssumptionsl=0_iii}), from (\ref{Eq5}) we arrive at (\ref{Eq8}) and (\ref{Eq9}).

Taking into account  (\ref{Eq8}), (\ref{Eq9}), as well as {\normalfont\textbf{Assumptions}} \hyperref[KAssumptions]{$K$}, {\normalfont\textbf{Assumptions}} \hyperref[GAssumptionsGeneric]{$G$} and the Assumption (\ref{ExtraAssumptionsl=0_iii}), from (\ref{Eq5}) we get for $\ae(x,t)\in X\times (0,T_0)$ that
\begin{equation}
\begin{array}{c}
\displaystyle v_2(x,t)\overset{(\ref{gAssumptions_ii})}{\geq} \int\limits_0^t \int\limits_X K(x,y;t-s) h(y,s) G(v_1(y,s)+\beta_0)d\mu(y)ds \\
\displaystyle \overset{(\ref{ExtraAssumptionsl=0_iii})\wedge(\ref{Eq8})}{\geq} \int\limits_0^t p_1(s)\int\limits_X K(x,y;t-s)G\left(G(\xi-\beta+\beta_0)\int\limits_0^s p_1(\tau)d\tau +\beta_0\right)d\mu(y)ds \\
\displaystyle \overset{(\ref{KAssumptions_ii})}{=} \int\limits_0^t p_1(s)G\left(G(\xi-\beta+\beta_0)\int\limits_0^sp_1(\tau)d\tau +\beta_0\right)ds,
\label{Eq10}
\end{array}
\end{equation}
while for $\ae(x,t) \in X \times[T_0, +\infty)$ we obtain

\begin{equation}
\begin{array}{c}
\displaystyle v_2(x,t)\overset{(\ref{gAssumptions_i})}{\geq} \int\limits_0^t \int\limits_X K(x,y;t-s)h(y,s)G(v_1(y,s))d\mu(y)ds \\
\displaystyle \overset{(\ref{Eq9})}{\geq}  \int\limits_0^t \int\limits_X K(x,y;t-s)h(y,s)G\left(G(\xi-\beta)\int\limits_0^s p_1(\tau)d\tau\right)d\mu(y)ds \\
\displaystyle \overset{(\ref{KAssumptions_ii})}{\geq} \int\limits_0^t p_1(s)G\left(G(\xi-\beta)\int\limits_0^s p_1(\tau)d\tau\right)ds \geq \int\limits_0^{T_0}p_1(s)G(G(\xi-\beta)\int\limits_0^s p_1(\tau)d\tau)ds.
\label{Eq11}
\end{array}
\end{equation}
On the other hand, due to the Assumptions (\ref{ExtraAssumptionsl=0_i}), (\ref{ExtraAssumptionsl=0_iii}) and (\ref{ExtraAssumptionsl=0_v}), and {\normalfont\textbf{Assumptions}} \hyperref[GAssumptionsGeneric]{$G$}, the right hand side of the inequality (\ref{Eq11}) can be estimated as follows,
\begin{equation}
\begin{array}{c}
\displaystyle 0<\int\limits_0^{T_0}p_1(s)G\left(G(\xi-\beta)\int\limits_0^sp_1(\tau)d\tau\right)ds < \int\limits_0^{\infty}p_2(s)G\left(G(\xi-\beta)\int\limits_0^sp_1(\tau)d\tau\right)ds \\
\displaystyle \leq \int\limits_0^{\infty}p_2(s)ds\, G\left(G(\xi-\beta)\right)<G(G(\xi))=G(\xi-\beta)<G(\xi)=\xi-\beta,  \label{Eq12}
\end{array}
\end{equation}
where we used the Assumption (\ref{ExtraAssumptionsl=0_i}) implying
$$
\int\limits_0^sp_1(\tau)d\tau\le\int\limits_0^sp_2(\tau)d\tau<\int\limits_0^{+\infty}p_2(\tau)d\tau\le1.
$$
We introduce the function $\mathfrak{L}:\mathbb{R}_+\to\mathbb{R}$ as
\begin{equation}
 \mathfrak{L}(t)\doteq \frac{\int\limits_0^tp_1(s)G\left(G(\xi-\beta+\beta_0)\int\limits_0^sp_1(\tau)d\tau+\beta_0\right)ds}{(\xi-\beta)\int\limits_0^tp_2(s)ds}, \quad \forall t\in\mathbb{R}_+. \label{Eq13}
\end{equation}
Now we shall study some properties of the function $\mathfrak{L}$. Note that from the Assumptions (\ref{ExtraAssumptionsl=0_i}), (\ref{ExtraAssumptionsl=0_ii}), (\ref{ExtraAssumptionsl=0_iii}) and (\ref{ExtraAssumptionsl=0_v}), and {\normalfont\textbf{Assumptions}} \hyperref[GAssumptionsGeneric]{$G$}, it follows that
\begin{equation}
\mathfrak{L}\in C(\mathbb{R}_+,\mathbb{R}_+), \label{Eq14}
\end{equation}
\begin{equation}
\begin{array}{c}
\displaystyle  \mathfrak{L}(t)\overset{(\ref{ExtraAssumptionsl=0_iii})}{\leq} \frac{\int\limits_0^tp_2(s)G\left(G(\xi-\beta+\beta_0) \int\limits_0^sp_2(\tau)d\tau+\beta_0\right)ds}{(\xi-\beta)\int\limits_0^t p_2(s)ds} \\
\displaystyle  \overset{(\ref{ExtraAssumptionsl=0_i})}{\leq} \frac{\int\limits_0^tp_2(s)ds\, G\left(G(\xi-\beta+\beta_0)+\beta_0\right)}{(\xi-\beta)\int\limits_0^t p_2(s)ds}\overset{(\ref{GAssumptionsGeneric_ii})}{<}\frac{G(G(\xi)+\beta_0)}{\xi-\beta}\\
\displaystyle \le\frac{G(G(\xi)+\beta)}{\xi-\beta}=\frac{G(\xi)}{\xi-\beta}=1, \quad \forall t\in \mathbb{R}_+,
\label{Eq15}
\end{array}
\end{equation}

\begin{equation}
\begin{array}{c}
\displaystyle 0<\lim\limits_{t\to+\infty}\mathfrak{L}(t)=\frac{1}{\xi-\beta}\int\limits_0^{\infty}p_1(s)G\left(G(\xi-\beta+\beta_0)\int\limits_0^s p_1(\tau)d\tau+\beta\right)ds \\
\displaystyle<\frac{G(G(\xi-\beta+\beta_0)+\beta)}{\xi-\beta}<\frac{G(\xi)}{\xi-\beta}=1.
\label{Eq16}
\end{array}
\end{equation}
Applying L'H\^opital's rule,
\begin{equation}
\begin{array}{c}
\lim\limits_{t\rightarrow 0^+}\mathfrak{L}(t)=\lim\limits_{t\rightarrow 0^+}\frac{\int\limits_0^tp_1(s)G(G(\xi-\beta+\beta_0)\int\limits_0^s p_1(\tau)d\tau+\beta_0)ds}{(\xi-\beta)\int\limits_0^tp_2(s)ds} \\
\displaystyle = \lim\limits_{t\rightarrow 0^+}\frac{p_1(t)G(G(\xi-\beta+\beta_0)\int\limits_0^t p_1(\tau)d\tau+\beta_0)}{(\xi-\beta)p_2(t)}\\
\displaystyle  =\lim\limits_{t\rightarrow 0^+}\frac{p_1(t)}{p_2(t)}\cdot \lim\limits_{t\rightarrow 0^+}\frac{G\left(G(\xi-\beta+\beta_0)\int\limits_0^t p_1(\tau)d\tau+\beta_0\right)}{\xi-\beta} \\
\displaystyle
 =\frac{G(\beta_0)}{\xi-\beta}\cdot \lim\limits_{t\rightarrow 0^+}\frac{p_1(t)}{p_2(t)}\overset{(\ref{ExtraAssumptionsl=0_ii})}{\in} (0,1),
\label{Eq17}
\end{array}
\end{equation}
since $p_1(t)\leq p_2(t)$ and $G(\beta_0)<G(\xi)=\xi-\beta$ and $\lim\limits_{t\rightarrow 0^+}\frac{p_1(t)}{p_2(t)}>0.$

Taking into account the listed properties (\ref{Eq16}) and (\ref{Eq17}), the function $\mathfrak{L}$ can be extended continuously to the domain $[0, +\infty)$, and we can state that there exists a number $\sigma_1\in(0,1)$ such that
\begin{equation}
 \mathfrak{L}(t)\geq \sigma_1, \quad\forall t\in [0, +\infty). \label{Eq18}
\end{equation}
Thus, by virtue of (\ref{Eq10}), (\ref{Eq13}), (\ref{Eq18}), and the following easy-to-check inequality,
\begin{equation}
v_1(x,t)\leq (\xi-\beta)\int\limits_0^tp_2(s)ds, \quad \ae x\in X, \quad \ae t\in \mathbb{R}_+, \label{Eq19}
\end{equation}
we arrive at the two-sided estimate
\begin{equation}
\sigma_1\leq \frac{v_2(x,t)}{v_1(x,t)}\leq 1,\quad\ae x\in X,\quad\ae t\in (0,T_0). \label{Eq20}
\end{equation}
On the other hand, due to (\ref{Eq11}) and (\ref{Eq19}), the following inequality holds true,
\begin{multline}
\sigma_2\doteq \frac{1}{\xi-\beta}\int\limits_0^{T_0} p_1(s)G\left(G(\xi-\beta)\int\limits_0^s p_1(\tau)d\tau\right)ds \\
\leq \frac{v_2(x,t)}{v_1(x,t)}\leq 1,\quad\ae x\in X,\quad\ae t\in[T_0, +\infty), \label{Eq21}
\end{multline}
and $\sigma_2 \in (0,1)$ by virtue of inequality (\ref{Eq12}).

Denote by
\begin{equation}
\sigma^*\doteq \min\{\sigma_1, \sigma_2\}\in (0,1). \label{Eq22}
\end{equation}
Then, taking into consideration (\ref{Eq20}), (\ref{Eq21}) and (\ref{Eq22}), we come to the inequalities
\begin{equation}
\sigma^*v_1(y,s)\leq v_2(y,s)\leq v_1(y,s), \quad \ae y\in X, \quad \ae s \in \mathbb{R}_+, \label{Eq23}
\end{equation}
whence it follows that
\begin{multline}
\sigma^* (v_1(y,s)+g(y,s))\leq v_2(y,s)+g(y,s) \\
\leq v_1(y,s)+g(y,s), \quad \ae y\in X, \quad  \ae s\in \mathbb{R}_+. \label{Eq24}
\end{multline}
Applying $G$ to all sides of the inequality (\ref{Eq24}), thanks to its monotonicity, we have
\begin{multline*}
G\left(\sigma^*(v_1(y,s)+g(y,s))\right)\leq G\left(v_2(y,s)+g(y,s)\right)\\
\leq G(v_1(y,s)+g(y,s)),\quad\ae y\in X,\quad\ae s\in\mathbb{R}.
\end{multline*}
Due to the Assumptions (\ref{ExtraAssumptionsl=0_iii}) and (\ref{ExtraAssumptionsl=0_vi}), and the fact that
$$
0\leq v_1(y,s)+g(y,s)\overset{(\ref{Eq19})}{\leq}\xi-\beta+g(y,s)\leq\xi,\quad\ae y\in X,\quad\ae s\in\mathbb{R}_+,
$$
we can write
$$
\varphi(\sigma^*)\int\limits_0^t\int\limits_XK(x,y,t-s)h(y,s)G(v_1(y,s)+g(y,s))d\mu(y)ds
$$
$$
\leq\int\limits_0^t\int\limits_XK(x,y,t-s)h(y,s)G(v_2(y,s)+g(y,s))d\mu(y)ds
$$
$$
\leq\int\limits_0^t\int\limits_XK(x,y,t-s)h(y,s)G(v_1(y,s)+g(y,s))d\mu(y)ds,\quad\ae x\in X,\quad\ae t\in\mathbb{R}_+,
$$
which, according to (\ref{Eq5}), means
$$\varphi(\sigma^*)v_2(x,t)\leq v_3(x,t)\leq v_2(x,t), \quad \ae x\in X, \quad  \ae t\in \mathbb{R}_+,$$
whence, in particular, it follows that
\begin{multline}
\varphi(\sigma^*)(v_2(y,s)+g(y,s))\leq v_3(y,s)+g(y,s)\\
\leq v_2(y,s)+g(y,s), \quad\ae y\in X, \quad  \ae s \in \mathbb{R}_+ . \label{Eq25}
\end{multline}
Due to (\ref{Eq25}), {\normalfont\textbf{Assumptions}} \hyperref[KAssumptions]{$K$}, {\normalfont\textbf{Assumptions}} \hyperref[gAssumptions]{$g$}, {\normalfont\textbf{Assumptions}} \hyperref[l*AssumptionsGeneric]{$h$}, {\normalfont\textbf{Assumptions}} \hyperref[GAssumptionsGeneric]{$G$}, the Assumption (\ref{ExtraAssumptionsl=0_iv}) and the non-negativity of the function $h$, from (\ref{Eq5}) we arrive at the inequality
$$\varphi(\varphi(\sigma^*))v_3(x,t)\leq v_4(x,t)\leq v_3(x,t), \quad \ae x\in X, \quad  \ae t\in \mathbb{R}_+.$$
Continuing this process, at the $m$-th step we obtain the following two sided inequality,
\begin{equation}
\underbrace{\varphi(\varphi\dots\varphi(\sigma^*))}\limits_{m}v_{m+1}(x,t)\leq v_{m+2}(x,t)\leq v_{m+1}(x,t), \quad \ae x\in X, \quad \ae t\in \mathbb{R}_+. \label{Eq26}
\end{equation}
From (\ref{Eq26}) and (\ref{Eq7}), due to the definition of the zeroth approximation in iterations (\ref{Eq5}), we arrive at inequalities
\begin{multline}
0\leq v_{m+1}(x,t)-v_{m+2}(x,t) \\
\leq (\xi-\beta)\left[1-\underbrace{\varphi(\varphi\dots\varphi(\sigma^*))}\limits_{m}\right], \quad \forall m\in \mathbb{N}, \quad \ae x\in X, \quad \ae t\in \mathbb{R}_+. \label{Eq27}
\end{multline}
Let $\varepsilon\in(0,1)$ be an arbitrary number. We will now make use of the inequality (3.16) from \cite{khach}, namely,
\begin{equation}
1>\underbrace{\varphi(\varphi\dots\varphi(\sigma^*))}\limits_{m}\geq k^m \sigma^*+1-k^m, \quad \forall m\in \mathbb{N}, \label{Eq28}
\end{equation}
where
\begin{equation}
k=k(\varepsilon) \doteq  \frac{1-\varphi(\varepsilon\sigma^*)}{1-\varepsilon\sigma^*}. \label{Eq29}
\end{equation}
Thus, from (\ref{Eq27}) combined with (\ref{Eq28}) and (\ref{Eq29}), it follows that
\begin{multline}
0\leq v_{m+1}(x,t)-v_{m+2}(x,t)\\
\leq (\xi -\beta)(1-\sigma^*)k^m, \quad \forall m\in \mathbb{N}, \quad\ae  x\in X, \quad \ae t\in \mathbb{R}_+. \label{Eq30}
\end{multline}
Now we rewrite the inequality (\ref{Eq30}) for members $m+1, m+2, \dots, m+p:$
$$0\leq v_{m+2}(x,t)-v_{m+3}(x,t)\leq (\xi-\beta)(1-\sigma^*)k^{m+1}, \quad \ae x\in X, \quad  \ae t\in \mathbb{R}_+,$$
$$\vdots$$
$$0\leq v_{m+p+1}(x,t)-v_{m+p+2}(x,t)\leq (\xi-\beta)(1-\sigma^*)k^{m+p}, \quad \ae x\in X, \quad  \ae t\in \mathbb{R}_+.$$
Summing up the corresponding sides of all these inequalities, and using the inequality (\ref{Eq30}), we get
\begin{equation}
\begin{array}{c}
\displaystyle 0\leq v_{m+1}(x,t)-v_{m+p+2}(x,t)\leq (\xi -\beta)(1-\sigma^*)(1+k+\dots+k^p)k^m\leq \\
\displaystyle \leq \frac{(\xi-\beta)(1-\sigma^*)}{1-k}k^m, \quad \forall m\in\mathbb{N}, \quad \ae x\in X, \quad \ae t\in \mathbb{R}_+. \label{Eq31}
\end{array}
\end{equation}
From (\ref{Eq30}) it immediately follows that the sequence of measurable functions $\{v_m\}_{m=0}^{\infty}$ converges uniformly to the measurable on $X\times \mathbb{R}_+$ function $v$ as $m\rightarrow +\infty.$ In (\ref{Eq31}), tending the number $p$ to infinity, we arrive at the inequality
\begin{multline}
0\leq v_{m+1}(x,t)-v(x,t)\\
\leq \frac{(\xi-\beta) (1-\sigma^*)}{1-k}k^m, \quad \ae x\in X, \quad \ae t\in \mathbb{R}_+, \quad \forall m\in \mathbb{N}. \label{Eq33}
\end{multline}
This concludes the proof of the theorem. $\Box$
\end{proof}

\subsection{The substochastic case: $\lambda_\pm>0$}

Suppose now that $\lambda_\pm>0$ in {\normalfont\textbf{Assumptions}} \hyperref[KAssumptions]{$K$}, which we dubbed the substochastic case. This time we will operate with one or more of the following additional assumptions on the functions $h$ and $G$.

\begin{assumptionssubSt}[Additional]\label{ExtraAssumptionsl>0} We assume that
\begin{itemize}

\item[(i)]
\begin{equation}
0<\alpha\doteq\inf_{(x,t)\in X\times\mathbb{R}_+}h(x,t);\label{ExtraAssumptionsl>0_i}
\end{equation}

\item[(ii)] $\exists\,\eta>0$ such that
\begin{equation}
\gamma\,G(\eta)=(\eta-\beta)\lambda_-,\quad\beta\doteq\|g\|_{L^\infty(X\times\mathbb{R}_+)}.\label{ExtraAssumptionsl>0_ii}
\end{equation}

\item[(iii)]
\begin{equation}
\gamma\doteq\|h\|_{L^\infty(X\times\mathbb{R}_+)}<\infty.\label{ExtraAssumptionsl>0_iii}
\end{equation}

\end{itemize}
\end{assumptionssubSt}
\begin{remark} Note that if the function $G$ is so-called strongly concave, i.e.,
$$
\lim_{u\to+\infty}\frac{G(u)}{u}=0,
$$
then the Assumption (\ref{ExtraAssumptionsl>0_ii}) is automatically satisfied.
\end{remark}

The corresponding existence results is as follows.

\begin{theorem}\label{Existencel>0Theorem} Suppose that the following assumptions hold:
\begin{itemize}

\item {\normalfont\textbf{Assumptions}} \hyperref[KAssumptions]{$K$} with $\lambda_\pm>0$;

\item {\normalfont\textbf{Assumptions}} $g$: \hyperref[gAssumptions_i]{(i)} and \hyperref[gAssumptions_iii]{(iii)};

\item {\normalfont\textbf{Assumptions}} \hyperref[l*AssumptionsGeneric]{$h$};

\item {\normalfont\textbf{Assumptions}} \hyperref[GAssumptionsGeneric]{$G$};

\item {\normalfont\textbf{Assumptions}} $\lambda_\pm>0$: \hyperref[ExtraAssumptionsl>0_ii]{(ii)} and \hyperref[ExtraAssumptionsl>0_iii]{(iii)}.

\end{itemize}
Then
$$
\exists\,u\in L^\infty(X\times\mathbb{R}_+,[0,+\infty)),
$$
such that $u$ satisfies the equation (\ref{IntEq}). Moreover, $u$ is the almost everywhere limit of the decreasing sequence $\{u_m\}_{m=1}^\infty$ of non-negative measurable functions, given by the following iterative approximations,
$$
u_{m+1}(x,t)\doteq g(x,t)+\int\limits_0^t\int\limits_XK(x,y;t-s)h(y,s)G(u_m(y,s))d\mu(y)ds,
$$
$$
u_0(x,t)\doteq\eta-\beta+g(x,t),\quad\forall m\in\{0\}\cup\mathbb{N},\quad\ae x\in X,\quad\ae t\in\mathbb{R}_+,
$$
with $\beta$ as in (\ref{ExtraAssumptionsl>0_ii}).

If, in addition, the following assumptions are true:
\begin{itemize}

\item {\normalfont\textbf{Assumptions}} $g$: \hyperref[gAssumptions_ii]{(ii)};

\item {\normalfont\textbf{Assumptions}} $\lambda_\pm>0$: \hyperref[ExtraAssumptionsl>0_i]{(i)},

\end{itemize}
then there exist $C>0$ and $k\in(0,1)$ such that
$$
0\le u_{m+1}(x,t)-u(x,t)\le C\,k^m,\quad\ae x\in X,\quad\ae t\in\mathbb{R}_+,\quad\forall m\in\mathbb{N}.
$$
\end{theorem}
\begin{proof}
Here we will put
$$
v_m(x,t)\doteq u_m(x,t)-g(x,t), \quad \forall m\in\{0\}\cup\mathbb{N} , \quad \ae x\in X, \quad \ae t\in \mathbb{R}_+.
$$
By induction in $m$, it is easy to check that
\begin{equation}
v_m\, \text{is measurable on the set}\,\, X\times \mathbb{R}_+, \quad\forall m\in\{0\}\cup\mathbb{N}, \label{Eq35}
\end{equation}
\begin{equation}
v_m(x,t)\geq 0,  \quad \forall m\in\{0\}\cup\mathbb{N},  \quad \ae x\in X, \quad \ae t\in \mathbb{R}_+, \label{Eq36}
\end{equation}
\begin{equation}
v_{m+1}(x,t)\leq v_{m}(x,t),   \quad \forall m\in\{0\}\cup\mathbb{N},  \quad \ae x\in X, \quad \ae t\in \mathbb{R}_+.\label{Eq37}
\end{equation}
As in the proof of Theorem \ref{Existencel=0Theorem}, we establish the existence of the solution $u$ which is the almost everywhere limit (infimum) of the approximating sequence $\{u_m\}_{m=1}^\infty$.

Now using {\normalfont\textbf{Assumptions}} \hyperref[KAssumptions]{$K$}, {\normalfont\textbf{Assumptions}} \hyperref[gAssumptions]{$g$}, {\normalfont\textbf{Assumptions}} \hyperref[GAssumptionsGeneric]{$G$} and  {\normalfont\textbf{Assumptions}} \hyperref[ExtraAssumptionsl>0]{$\lambda_\pm>0$}, we will estimate iteration functions $v_1$ and $v_2$ from above and below. First,
 $$v_1(x,t)=\int\limits_0^t\int\limits_X K(x,y;t-s)h(y,s)G(v_0(y,s)+g(y,s))d\mu(y)ds $$
 $$\leq\gamma \int\limits_0^t \int\limits_X K(x,y; t-s)G(\eta-\beta+g(y,s))d\mu(y)ds$$
 $$\leq \gamma G(\eta)\int\limits_0^t e^{-(t-s)\lambda_-}ds =(\eta-\beta)(1-e^{-t\lambda_-}), \quad \ae x\in X, \quad \ae t\in \mathbb{R}_+.$$
Thus,
\begin{equation}
v_1(x,t)\leq (\eta-\beta)(1-e^{-t\lambda_-}), \quad \ae x\in X, \quad \ae t\in \mathbb{R}_+.\label{Eq38}
\end{equation}
On the other hand, for $\ae x\in X$ and $\ae t\in (0,T_0)$ we have
\begin{equation}
\begin{array}{c}
\displaystyle v_1(x,t)\overset{(\ref{ExtraAssumptionsl>0_i})}{\geq} \alpha \int\limits_0^t \int\limits_X K(x,y;t-s)G(\eta-\beta+g(y,s))d\mu(y)ds\\
\displaystyle \overset{(\ref{gAssumptions_ii})}{\geq} \alpha G(\eta-\beta +\beta_0)\int\limits_0^t e^{-(t-s)\lambda_+}ds =\frac{\alpha}{\lambda_+}G(\eta-\beta+\beta_0)(1-e^{-t\lambda_+}),
\label{Eq39}
\end{array}
\end{equation}
while for $\ae x\in X$ and $\ae t\in [T_0, +\infty)$ we have
\begin{equation}
\begin{array}{c}
\displaystyle v_1(x,t)\overset{(\ref{ExtraAssumptionsl>0_i})}{\geq} \alpha \int\limits_0^t \int\limits_X K(x,y;t-s)G(\eta-\beta+g(y,s))d\mu(y)ds \\
\displaystyle \overset{(\ref{KAssumptions_ii})}{\geq} \frac{\alpha}{\lambda_+}G(\eta-\beta)(1-e^{-t\lambda_+})\geq \frac{\alpha}{\lambda_+}G(\eta-\beta)(1-e^{-T_0\lambda_+}).
\label{Eq40}
\end{array}
\end{equation}
Therefore, for $\ae x\in X$, $\ae t\in (0,T_0)$ we have
\begin{equation}
\begin{array}{c}
\displaystyle v_2(x,t)\overset{(\ref{ExtraAssumptionsl>0_i})}{\geq} \alpha \int\limits_0^t \int\limits_X K(x,y;t-s)G(v_1(y,s)+g(y,s))d\mu(y)ds \\
\displaystyle \geq \alpha \int\limits_0^t \int\limits_XK(x,y;t-s)G(v_1(y,s)+\beta_0)d\mu(y)ds \\
\displaystyle \overset{(\ref{Eq40})}{\geq} \alpha \int\limits_0^t G\left(\frac{\alpha}{\lambda_+}G(\eta-\beta+\beta_0)(1-e^{-s\lambda_+})+\beta_0\right)e^{-(t-s)\lambda_+}ds\geq \frac{\alpha}{\lambda_+}G(\beta_0)(1-e^{-t\lambda_+}),
\label{Eq41}
\end{array}
\end{equation}
while for $\ae x\in X$, $\ae t\in [T_0, +\infty)$ we get
\begin{equation}
\begin{array}{c}
\displaystyle v_2(x,t)\geq \alpha \int\limits_0^t \int\limits_X K(x,y;t-s)G(v_1(y,s))d\mu(y)ds \\
\displaystyle \overset{(\ref{Eq40})}{\geq}  \alpha  G\left(\frac{\alpha}{\lambda_+}G(\eta-\beta)(1-e^{-T_0\lambda_+})\right)\int\limits_0^t\int\limits_X K(x,y;t-s)d\mu(y)ds\\
\displaystyle\geq\frac{\alpha}{\lambda_+}G\left(\frac{\alpha}{\lambda_+}G(\eta-\beta)(1-e^{-T_0\lambda_+})\right)(1-e^{-t\lambda_+}).
\label{Eq42}
\end{array}
\end{equation}
In particular, from the last inequality it follows that
\begin{multline}
v_2(x,t)\geq  \frac{\alpha}{\lambda_+}G\left(\frac{\alpha}{\lambda_+}G(\eta-\beta)(1-e^{-T_0\lambda_+})\right)(1-e^{-T_0\lambda_+}), \\
\quad \ae x\in X, \quad \ae t\in [T_0,+\infty).
\label{Eq43}
\end{multline}
Thus, considering (\ref{Eq37}), (\ref{Eq38}), (\ref{Eq39}) and (\ref{Eq41}), for $\ae x\in X$, $\ae t\in (0, T_0)$ we arrive at the following two-sided estimate:
\begin{equation}
\frac{\alpha G(\beta_0)}{(\eta-\beta)\lambda_-}\frac{\lambda_-}{\lambda_+}<\frac{\alpha G(\beta_0)}{(\eta-\beta)\lambda_+}\frac{1-e^{-t\lambda_+}}{1-e^{-t\lambda_-}}\leq \frac{v_2(x,t)}{v_1(x,t)}\leq 1.
\label{Eq44}
\end{equation}
Now, due to (\ref{Eq37}), (\ref{Eq38}) and  (\ref{Eq42}),  for $\ae x\in X$, $\ae t\in [T_0,+\infty)$ we confirm the validity of the following double estimate,
\begin{eqnarray}
\frac{\alpha}{(\eta-\beta)\lambda_-}G\left(\frac{\alpha}{\lambda_+}G(\eta-\beta)(1-e^{-T_0\lambda_+})\right)\frac{\lambda_-}{\lambda_+}\nonumber\\
<\frac{\alpha}{(\eta-\beta)\lambda_+}G\left(\frac{\alpha}{\lambda_+}G(\eta-\beta)(1-e^{-T_0\lambda_+})\right)\frac{1-e^{-t\lambda_+}}{1-e^{-t\lambda_-}}\leq \frac{v_2(x,t)}{v_1(x,t)}\leq 1.
\label{Eq45}
\end{eqnarray}
Now let us show that
\begin{equation}
\sigma^{\sharp}\doteq\min\left\{\frac{\alpha G(\beta_0)}{(\eta-\beta)\lambda_-}, \quad  \frac{\alpha}{(\eta-\beta)\lambda_-}G\left(\frac{\alpha}{\lambda_+}G(\eta-\beta)(1-e^{-T_0\lambda_+})\right)\right\}\in (0,1).
\label{Eq46}
\end{equation}
Indeed, by virtue of {\normalfont\textbf{Assumptions}} \hyperref[GAssumptionsGeneric]{$G$} and {\normalfont\textbf{Assumptions}} \hyperref[ExtraAssumptionsl>0]{$\lambda_\pm>0$}, we have
$$0<\frac{\alpha G(\beta_0)}{(\eta-\beta)\lambda_-}\leq \frac{\alpha G(\beta)}{(\eta-\beta)\lambda_-}< \frac{\alpha G(\eta)}{(\eta-\beta)\lambda_-}=\frac{\alpha}{\gamma}\leq 1,$$
$$0<\frac{\alpha}{(\eta-\beta)\lambda_-}G\left(\frac{\alpha}{\lambda_+}G(\eta-\beta)(1-e^{-T_0\lambda_+})\right)< \frac{\alpha}{(\eta-\beta)\lambda_-}G\left(\frac{\alpha}{\lambda_-}G(\eta-\beta)\right)$$
$$<\frac{\alpha}{(\eta-\beta)\lambda_-}G\left(\frac{\alpha}{\lambda_-}G(\eta)\right)<\frac{\alpha}{(\eta-\beta)\lambda_-}G\left(\frac{\alpha}{\gamma}(\eta-\beta)\right)<\frac{\alpha}{(\eta-\beta)\lambda_-}G(\eta)=\frac{\alpha}{\gamma}\leq 1.$$
Denote
$$
\sigma^{\#}\doteq\frac{\lambda_-}{\lambda_+}\sigma^{\sharp}\in(0,1).
$$
Thus, taking into account the equations (\ref{Eq44})-(\ref{Eq46}), we conclude that
$$
\sigma^{\#}v_1(x,t)\leq v_2(x,t)\leq v_1(x,t), \quad \ae x\in X, \quad \ae t\in \mathbb{R}_+.
$$
Further, following a reasoning similar to the one in the proof of Theorem \ref{Existencel=0Theorem}, we complete the proof of the formulated theorem. $\Box$
\end{proof}

\subsection{Relaxing assumptions on $g$}

In view of the interpretation (\ref{AuxCauchyProblem}) we have

$$
u_0(x)=\lim\limits_{t\to0+}g(x,t).
$$
Thus, the assumption {\normalfont\textbf{Assumptions}} $g$: \hyperref[gAssumptions_ii]{(ii)} implies the restriction $\inf u_0>0$ on the Cauchy data, which may be problematic for applications in PDEs. In particular, if the measure $\mu$ is infinite, then $u_0\notin L^p(X,\mu)$ for any $p\in[1,+\infty)$. A vertical shifting argument may help translating this Cauchy problem to a slightly different one without an infimum condition on the Cauchy data.

Consider the following alternative Cauchy problem,
\begin{equation}
\begin{cases}
\partial_t v(x,t)+\mathcal{L}v(x,t)=h(x,t)\tilde G(u(x,t))+\tilde f(x,t),\\
v(x,0)=v_0(x),
\end{cases}\label{CauchyProblemAlt}
\end{equation}
and the corresponding integral equation
\begin{equation}
v(x,t)=\tilde g(x,t)+\int\limits_0^t\int\limits_XK(x,y;t-s)h(y,s)\tilde G(v(y,s))d\mu(y)ds.\label{IntEqAlt}
\end{equation}
These are related through the auxiliary Cauchy problem
$$
\begin{cases}
\partial_t\tilde g(x,t)+\mathcal{L}\tilde g(x,t)=\tilde f(x,t),\\
\tilde g(x,0)=v_0(x).
\end{cases}
$$

Consider the following assumptions for the new functions $\tilde g$ and $\tilde G$.

\begin{assumptionsgAlt}[Term $\tilde g$ in (\ref{IntEqAlt})]\label{gAssumptionsAlt} We assume that
\begin{itemize}

\item[(i\'\,\!)]
\begin{equation}
\tilde g:X\times\mathbb{R}_+\to[0,+\infty)\quad\mbox{is measurable};\label{gAssumptionsAlt_i}
\end{equation}

\item[(iii\'\,\!)]
\begin{equation}
\tilde g\in L^\infty(X\times\mathbb{R}_+).\label{gAssumptionsAlt_iii}
\end{equation}

\end{itemize}
\end{assumptionsgAlt}

\begin{assumptionsGAlt}[Non-linearity $\tilde G$]\label{GAssumptionsGenericAlt} We assume that
\begin{itemize}

\item[(i\'\,\!)] For some $\beta_0>0$,
\begin{equation}
\tilde G\in C([-\beta_0,+\infty),[0,+\infty));\label{GAssumptionsGenericAlt_i}
\end{equation}

\item[(ii\'\,\!)]
\begin{equation}
\tilde G(-\beta_0)=0,\quad\tilde G\quad\mbox{is strictly increasing and concave}.\label{GAssumptionsGenericAlt_ii}
\end{equation}

\end{itemize}
\end{assumptionsGAlt}

Let us also introduce the modified additional assumptions for the stochastic and substochastic cases, correspondingly.

\begin{assumptionsStAlt}[Additional]\label{ExtraAssumptionsl=0Alt} We assume that $\exists\,p_1,p_2\in C(\mathbb{R}_+,\mathbb{R}_+)$ such that
\begin{itemize}

\item[(i\'\,\!)]
\begin{equation}
p_1\not\equiv p_2,\quad\int\limits_0^\infty p_2(t)dt=1;\label{ExtraAssumptionsl=0Alt_i}
\end{equation}

\item[(ii\'\,\!)]
\begin{equation}
\lim_{t\to0+}\,\frac{p_1(t)}{p_2(t)}>0;\label{ExtraAssumptionsl=0Alt_ii}
\end{equation}

\item[(iii\'\,\!)]
\begin{equation}
p_1(t)\le h(x,t)\le p_2(t),\quad\ae x\in X,\quad\ae t\in\mathbb{R}_+.\label{ExtraAssumptionsl=0Alt_iii}
\end{equation}

\end{itemize}
Suppose $\exists\,\varphi:[0,1]\to[0,1]$ homeomorphism such that
\begin{itemize}

\item[(iv\'\,\!)]
\begin{equation}
\varphi(0)=0,\quad\varphi(1)=1,\quad\varphi\quad\mbox{concave};\label{ExtraAssumptionsl=0Alt_iv}
\end{equation}

\item[(v\'\,\!)] $\exists\,\tilde\xi\in(-\beta_0,+\infty)$ such that
\begin{equation}
\tilde\xi-\tilde G(\tilde\xi)=\|\tilde g\|_{L^\infty(X\times\mathbb{R}_+)};\label{ExtraAssumptionsl=0Alt_v}
\end{equation}

\item[(vi\'\,\!)]
\begin{equation}
\tilde G(\sigma v-(1-\sigma)\beta_0)\ge\varphi(\sigma)\tilde G(v),\quad\forall\sigma\in[0,1],\quad\forall v\in[-\beta_0,\tilde\xi].\label{ExtraAssumptionsl=0Alt_vi}
\end{equation}

\quad

\end{itemize}
\end{assumptionsStAlt}

\begin{assumptionssubStAlt}[Additional]\label{ExtraAssumptionsl>0Alt} We assume that
\begin{itemize}

\item[(i\'\,\!)]
\begin{equation}
0<\alpha\doteq\inf_{(x,t)\in X\times\mathbb{R}_+}h(x,t);\label{ExtraAssumptionsl>0Alt_i}
\end{equation}

\item[(ii\'\,\!)] $\exists\,\tilde\eta\in(-\beta_0,+\infty)$ such that
\begin{equation}
\gamma\,G(\tilde\eta)=(\tilde\eta-\tilde\beta)\lambda_-,\quad\tilde\beta\doteq\|\tilde g\|_{L^\infty(X\times\mathbb{R}_+)}.\label{ExtraAssumptionsl>0Alt_ii}
\end{equation}

\item[(iii\'\,\!)]
\begin{equation}
\gamma\doteq\|h\|_{L^\infty(X\times\mathbb{R}_+)}<\infty.\label{ExtraAssumptionsl>0Alt_iii}
\end{equation}

\end{itemize}
\end{assumptionssubStAlt}

\begin{remark} Note that if the function $\tilde G$ is so-called strongly concave, i.e.,
$$
\lim_{v\to+\infty}\frac{\tilde G(v)}{v}=0,
$$
then the Assumptions (\ref{ExtraAssumptionsl=0Alt_v}) and (\ref{ExtraAssumptionsl>0Alt_ii}) are automatically satisfied.
\end{remark}

Now we can state the existence results analogous to Theorem \ref{Existencel=0Theorem} and Theorem \ref{Existencel>0Theorem}.

\begin{corollary}\label{Existencel=0Corr}  Suppose that the following assumptions hold:
\begin{itemize}

\item {\normalfont\textbf{Assumptions}} \hyperref[KAssumptions]{$K$} hold with $\lambda_\pm=0$;

\item {\normalfont\textbf{Assumptions}} \hyperref[gAssumptionsAlt]{$\tilde g$};

\item {\normalfont\textbf{Assumptions}} \hyperref[l*AssumptionsGeneric]{$h$};

\item {\normalfont\textbf{Assumptions}} \hyperref[GAssumptionsGenericAlt]{$\tilde G$};

\item {\normalfont\textbf{Assumptions}} \hyperref[ExtraAssumptionsl=0Alt]{$\widetilde{\lambda_\pm=0}$}.

\end{itemize}
Then
$$
\exists\,v\in L^\infty(X\times\mathbb{R}_+,[0,+\infty)),
$$
such that $v$ satisfies the equation (\ref{IntEqAlt}). Moreover, $v$ is the uniform limit of the decreasing sequence $\{v_m\}_{m=1}^\infty$ of non-negative measurable functions, given by following iterative approximations,
$$
v_{m+1}(x,t)\doteq\tilde g(x,t)+\int\limits_0^t\int\limits_XK(x,y;t-s)h(y,s)\tilde G(v_m(y,s))d\mu(y)ds,
$$
$$
v_0(x,t)\doteq\tilde\xi-\tilde\beta+\tilde g(x,t),\quad\forall m\in\{0\}\cup\mathbb{N},\quad\ae x\in X,\quad\ae t\in\mathbb{R}_+,
$$
with $\tilde\beta\doteq\|\tilde g\|_{L^\infty(X\times\mathbb{R}_+)}$. In fact, there exist $C>0$ and $k\in(0,1)$ such that
$$
0\le v_{m+1}(x,t)-v(x,t)\le C\,k^m,\quad\ae x\in X,\quad\ae t\in\mathbb{R}_+,\quad\forall m\in\mathbb{N}.
$$
\end{corollary}
\begin{proof} Let $\beta_0>0$ be as in {\normalfont\textbf{Assumptions}} \hyperref[GAssumptionsGenericAlt]{$\tilde G$}, and define $g\doteq\tilde g+\beta_0$ and $G:[0,+\infty)\to[0,+\infty)$ by $G(u)\doteq\tilde G(u-\beta_0)$ for all $u\in[0,+\infty)$. It is then easy to verify that all assumptions appearing in Theorem \ref{Existencel=0Theorem} are satisfied, producing a solution $u$ to the equation (\ref{IntEq}) together with the uniform approximation scheme $\{u_m\}_{m=1}^\infty$. It remains to check that $v\doteq u-\beta_0$ is a solution to the equation (\ref{IntEqAlt}), and $\{v_m\}_{m=1}^\infty$ is the desired approximating sequence with $v_m\doteq u_m-\beta_0$ for $m\in\mathbb{N}$. This completes the proof.
\end{proof}

\begin{corollary}\label{Existencel>0Corr} Suppose that the following assumptions hold:
\begin{itemize}

\item {\normalfont\textbf{Assumptions}} \hyperref[KAssumptions]{$K$} with $\lambda_\pm>0$;

\item {\normalfont\textbf{Assumptions}} \hyperref[gAssumptionsAlt]{$\tilde g$};

\item {\normalfont\textbf{Assumptions}} \hyperref[l*AssumptionsGeneric]{$h$};

\item {\normalfont\textbf{Assumptions}} \hyperref[GAssumptionsGenericAlt]{$\tilde G$};

\item {\normalfont\textbf{Assumptions}} \hyperref[ExtraAssumptionsl>0Alt]{$\widetilde{\lambda_\pm>0}$}.

\end{itemize}
Then
$$
\exists\,v\in L^\infty(X\times\mathbb{R}_+,[0,+\infty)),
$$
such that $v$ satisfies the equation (\ref{IntEqAlt}). Moreover, $v$ is the uniform limit of the decreasing sequence $\{v_m\}_{m=1}^\infty$ of non-negative measurable functions, given by the following iterative approximations,
$$
v_{m+1}(x,t)\doteq\tilde g(x,t)+\int\limits_0^t\int\limits_XK(x,y;t-s)h(y,s)\tilde G(v_m(y,s))d\mu(y)ds,
$$
$$
v_0(x,t)\doteq\tilde\eta-\tilde\beta+\tilde g(x,t),\quad\forall m\in\{0\}\cup\mathbb{N},\quad\ae x\in X,\quad\ae t\in\mathbb{R}_+,
$$
with $\tilde\beta$ as in (\ref{ExtraAssumptionsl>0Alt_ii}). In fact, there exist $C>0$ and $k\in(0,1)$ such that
$$
0\le v_{m+1}(x,t)-v(x,t)\le C\,k^m,\quad\ae x\in X,\quad\ae t\in\mathbb{R}_+,\quad\forall m\in\mathbb{N}.
$$
\end{corollary}
\begin{proof} The proof is very similar to that of Corollary \ref{Existencel=0Corr} above, but now based on Theorem \ref{Existencel>0Theorem}.
\end{proof}

The advantage of (\ref{IntEqAlt}) over (\ref{IntEq}) is that in the corresponding Cauchy problem (\ref{CauchyProblemAlt}), $\inf v_0$ does not have to be strictly positive, in contrast to (\ref{CauchyProblem}). This comes at the expense of perhaps less natural structure of the non-linearity $\tilde G$ as required in {\normalfont\textbf{Assumptions}} \hyperref[GAssumptionsGenericAlt]{$\tilde G$}.

\section{Uniqueness of solutions}

In this section we will show that the solutions obtained in the previous section are unique.

\begin{theorem}\label{Uniquenessl=0Theorem} Suppose that all assumptions appearing in Theorem \ref{Existencel=0Theorem} are true. Then the equation (\ref{IntEq}) cannot have two distinct non-negative solutions in $L^\infty(X\times\mathbb{R}_+)$.
\end{theorem}
\begin{proof}
We retain the terminology and notations introduced in the proof of Theorem \ref{Existencel=0Theorem}. Since the equations (\ref{IntEq}) and (\ref{Eq34}) are equivalent, we will present a proof of the uniqueness of the solution for the equation (\ref{Eq34}).
We assume that the equation (\ref{Eq34}), besides the solution $v$ (constructed by means of successive approximations (\ref{Eq5})), has also another non-negative, measurable solution $v^*$ bounded on $X\times \mathbb{R}_+$. First, taking into account {\normalfont\textbf{Assumptions}} \hyperref[GAssumptionsGeneric]{$G$} and the Assumptions (\ref{gAssumptions_ii}) and (\ref{ExtraAssumptionsl=0_iii}), we estimate the second solution $v^*$ from below. For $\ae x\in X$, $\ae t\in (0, T_0)$, we have
\begin{equation}
v^*(x,t)\geq \int\limits_0^t\int\limits_X K(x,y;t-s)h(y,s)G(g(y,s))d\mu(y)ds\overset{(\ref{gAssumptions_ii}),(\ref{ExtraAssumptionsl=0_iii})}{\geq} G(\beta_0)\int\limits_0^t p_1(s)ds,
\label{Eq48}
\end{equation}
while for $\ae x\in X$, $\ae t\in [T_0,+\infty)$,
\begin{equation}
\begin{array}{c}
\displaystyle v^*(x,t)\overset{(\ref{ExtraAssumptionsl=0_iii})}{\geq} \int_0^t\limits p_1(s) \int\limits_X K(x,y;t-s)G(g(y,s))d\mu(y)ds \\
\displaystyle \overset{(\ref{gAssumptions_ii})}{\geq} \int\limits_0^{T_0}p_1(s)\int\limits_X K(x,y;t-s)G(g(y,s))d\mu(y)ds\geq G(\beta_0)\int\limits_0^{T_0}p_1(s)ds.
\label{Eq49}
\end{array}
\end{equation}
Now we prove the validity of the following inequality:
\begin{equation}
v^*(x,t)\leq v(x,t), \quad \ae x\in X, \quad \ae t\in \mathbb{R}_+.
\label{Eq50}
\end{equation}
For this purpose we will first prove the validity of the following estimate from above,
\begin{equation}
v^*(x,t)\leq \xi-\beta, \quad \ae  x\in X, \quad \ae  t\in \mathbb{R}_+.
\label{Eq51}
\end{equation}
Denote by
\begin{equation}
0<C^* \doteq\sup\limits_{(x,t)\in X\times \mathbb{R}_+} v^*(x,t).
\label{Eq52}
\end{equation}
Then, from (\ref{Eq34}) combined with {\normalfont\textbf{Assumptions}} \hyperref[GAssumptionsGeneric]{$G$} and the Assumptions (\ref{ExtraAssumptionsl=0_i}) and (\ref{ExtraAssumptionsl=0_iii}), it follows that
$$0\leq v^*(x,t)\leq \int\limits_0^t\int\limits_X K(x,y;t-s)h(y,s)G(C^*+g(y,s))d\mu(y)ds$$
$$\leq G(C^*+\beta)\int\limits_0^t p_2(s)ds<G(C^*+\beta), \quad \ae  x\in X, \quad \ae  t\in \mathbb{R}_+,$$
whence we immediately get
\begin{equation}
C^*\leq G(C^*+\beta).
\label{Eq53}
\end{equation}
Let us show that
\begin{equation}
C^*\leq \xi-\beta.
\label{Eq54}
\end{equation}
Suppose the opposite is true: $C^*>\xi-\beta.$ Then, taking into account the fact that the function $u\mapsto\dfrac{G(u)}{u}$ is monotonically decreasing on $\mathbb{R}_+$ (which immediately follows from {\normalfont\textbf{Assumptions}} \hyperref[GAssumptionsGeneric]{$G$}), we get
$$\frac{G(C^*+\beta)}{C^*+\beta}\le\frac{G(\xi)}{\xi}=\frac{\xi-\beta}{\xi},$$
whence it follows that
\begin{equation}
G(C^*+\beta)\le\left(1-\frac{\beta}{\xi}\right)(C^*+\beta)=C^*+\beta-\frac{C^*\beta}{\xi}-\frac{\beta^2}{\xi}<C^*,
\label{Eq55}
\end{equation}
because, by assumption, $C^*>\xi-\beta.$ But the inequality (\ref{Eq55}) contradicts the inequality (\ref{Eq53}). Hence, the inequality  (\ref{Eq54}) holds true, implying the estimate (\ref{Eq51}). Now, using the inequality (\ref{Eq51}) and the formula (\ref{Eq5}), by induction in $m$, it is easy to prove the validity of the following inequalities:
\begin{equation}
v^*(x,t)\leq v_m(x,t), \quad \forall m\in \{0\}\cup \mathbb{N}, \quad \ae
 x\in X, \quad \ae  t\in \mathbb{R}_+.
\label{Eq56}
\end{equation}
In the inequality (\ref{Eq56}), tending $m$ to infinity, we arrive at the estimate (\ref{Eq50}). From (\ref{Eq32}), together with {\normalfont\textbf{Assumptions}} \hyperref[GAssumptionsGeneric]{$G$} and the Assumptions (\ref{ExtraAssumptionsl=0_iii}), (\ref{ExtraAssumptionsl=0_v}), it also follows that
\begin{equation}
\begin{array}{c}
\displaystyle v(x,t)\leq \int\limits_0^t\int\limits_X K(x,y;t-s)h(y,s)G(\xi-\beta+g(y,s))d\mu(y)ds \\
\displaystyle \leq (\xi-\beta)\int\limits_0^t p_2(s)ds, \quad \ae  x\in X, \quad  \ae  t\in \mathbb{R}_+.
\label{Eq57}
\end{array}
\end{equation}
Thus, by virtue of (\ref{Eq48}), (\ref{Eq50}) and (\ref{Eq57}), for $\ae x\in X$, $\ae t\in(0,T_0)$ we come to the inequalities
\begin{equation}
 \mathcal{B}(t)\leq \frac{v^*(x,t)}{v(x,t)}\leq 1,
\label{Eq58}
\end{equation}
where
\begin{equation}
 \mathcal{B}(t)\doteq\frac{G(\beta_0)\int\limits_0^t p_1(s)ds}{(\xi-\beta) \int\limits_0^t p_2(s)}ds, \quad \forall  t\in \mathbb{R}_+, \label{Eq59}
\end{equation}
whereas for $\ae x\in X$, $\ae t\in [T_0, +\infty)$, due to (\ref{Eq49}),  (\ref{Eq50}) and  (\ref{Eq32}), we get the chain of inequalities
\begin{equation}
\frac{G(\beta_0)\int\limits_0^{T_0}p_1(s)ds}{\xi-\beta}\leq \frac{v^*(x,t)}{v(x,t)}\leq 1.
\label{Eq60}
\end{equation}
From
(\ref{Eq59}), due to Assumptions (\ref{ExtraAssumptionsl=0_i}) and (\ref{ExtraAssumptionsl=0_ii}),  and applying L'Hôpital's rule, we have
\begin{equation}
\mathcal{B} \in C(\mathbb{R}_+), \quad \mathcal{B}(t)<1, \quad \forall  t\in \mathbb{R}_+,
\label{Eq61}
\end{equation}
\begin{equation}
\lim\limits_{t\rightarrow +\infty}\mathcal{B}(t)=\frac{G(\beta_0)}{\xi-\beta}\int\limits_0^{\infty}p_1(s)ds<\frac{G(\beta_0)}{\xi-\beta}\in(0,1), \label{Eq62}
\end{equation}
\begin{equation}
\lim\limits_{t\rightarrow 0^+}\mathcal{B}(t)=\frac{G(\beta_0)}{\xi-\beta}\lim\limits_{t\rightarrow 0^+}\frac{p_1(t)}{p_2(t)}\in (0,1). \label{Eq63}
\end{equation}
It is then obvious that the function $\mathcal{B}$ can be extended continuously to the domain $[0,+\infty]$, and one can say that there exists a number $\tilde{\sigma}_1\in (0,1),$ such that
\begin{equation}
\mathcal{B}(t)\geq \tilde{\sigma}_1, \quad\forall t\in [0,+\infty).\label{Eq64}
\end{equation}
We put
$$
\tilde{\sigma}_2 \doteq\frac{G(\beta_0)\int\limits_0^{T_0}p_1(s)ds}{\xi-\beta}\in (0,1).
$$
and denote by
$$\tilde{\sigma}\doteq\min\{\tilde{\sigma}_1,\tilde{\sigma}_2\}\in (0,1).$$
Then, taking into account (\ref{Eq58})-(\ref{Eq60}) and (\ref{Eq64}), we arrive at inequalities
\begin{equation}
\tilde{\sigma}v(x,t)\leq v^*(x,t)\leq v(x,t), \quad \ae  x\in X, \quad \ae  t\in \mathbb{R}_+.\label{Eq65}
\end{equation}
Further, using the Assumptions (\ref{ExtraAssumptionsl=0_iv}) and (\ref{ExtraAssumptionsl=0_vi}),
and performing steps similar to those in the proof of inequality (\ref{Eq30}), we come to the estimate
\begin{equation}
0\leq v(x,t)-v^*(x,t)\leq C_1 k_1^m,\quad \forall
 m\in \mathbb{N},  \quad \ae   x\in X, \quad \ae    t\in \mathbb{R}_+ ,\label{Eq66}
\end{equation}
where $C_1>0, \,\, k_1\in(0,1)$ are constants.

In the inequality (\ref{Eq66}), taking $m$ to infinity, we obtain that
$$
v(x,t)=v^*(x,t), \quad \ae x\in X, \quad \ae t\in \mathbb{R}_+.
$$
The theorem is proved. $\Box$
\end{proof}

\begin{theorem} Suppose that all assumptions appearing in Theorem \ref{Existencel>0Theorem} are true. Then the equation (\ref{IntEq}) cannot have two distinct non-negative solutions in $L^\infty(X\times\mathbb{R}_+)$.
\end{theorem}
\begin{proof} The proof is analogous to that of Theorem \ref{Uniquenessl=0Theorem}. $\Box$
\end{proof}

\section{Discussion of assumptions on the kernel $K$}
\label{SEC:K}

The above statements about the existence and uniqueness of solutions to the integral equation in question were made under many technical assumptions about various parts of the equation. The most decisive among these, in terms of the applicability of our methods in a wider context, are the assumptions on the kernel $K$ as in {\normalfont\textbf{Assumptions}} \hyperref[KAssumptions]{$K$}. Namely, that $K$ is non-negative and stochastic or substochastic. If we interpret solutions of the integral equation (\ref{IntEq}) as the mild solutions of the parabolic equation (\ref{CauchyProblem}), then the question is: for which operators $\mathcal{L}$, the heat semigroup $\{e^{-t\mathcal{L}}\}_{t\in\mathbb{R}_+}$ admits such a heat kernel $K$?

The most basic requirement is that $\mathcal{L}$ be a non-negative operator, in one sense or another, appropriate to the construction of the semigroup. The stochasticity condition in the Assumption (\ref{KAssumptions_ii}) can be formally written as $e^{-t\mathcal{L}}1=1$, where $1\in L^\infty(X)$ is the constant function equal to one. On the other hand, the (almost everywhere) positivity condition in the Assumption (\ref{KAssumptions_i}) can be formally written as $e^{-t\mathcal{L}}f\ge0$ for all $f\ge0$. These two conditions together essentially describe the definition of a \textit{Markov semigroup} $\{e^{-t\mathcal{L}}\}_{t\in\mathbb{R}_+}$.

There are several important contexts where the operator $\mathcal{L}$ does generate a Markov heat semigroup, some of which we recall below. We note that the only requirements on $\mathcal{L}$ are those that guarantee our {\normalfont\textbf{Assumptions}} \hyperref[KAssumptions]{$K$} on the corresponding  heat kernel $K$.
\begin{example}
$\mathcal{L}=-\Delta$ for the Laplace-Beltrami operator $\Delta$ on a complete Riemannian manifold with Ricci curvature bounded from below. Indeed, by Theorem 5.2.6 (i) in \cite{Dav89}, we have $e^{-t\mathcal{L}}1=1$. Then, by Theorem 5.2.1 in \cite{Dav89}, the kernel $K$ is smooth and strictly positive. In fact, the non-negativity of $K$ holds without any assumptions on completeness or curvature, see Theorem 7.13 in \cite{Gri09}. And the stochasticity holds for complete connected Riemannian manifolds without restrictions on the curvature, but with the volume growth condition
$$
\int\limits^\infty\frac{rdr}{\ln V(r)}=+\infty,
$$
which allows volume growth as rapid as $\exp(c\;r^2)$, see Theorem 11.8 in \cite{Gri09}.
\end{example}

\begin{example}
The stochasticity holds also for sub-Laplacian operators $\Delta_\mathrm{sub}$ on sub-Rieman\-nian manifolds satisfying so-called curvature-dimension inequalities, see Theorem 2.19 in \cite{BaGa17}, but also Theorem 6 in \cite{Mun11}. This includes sub-Riemannian manifolds with transverse symmetries.
\end{example}

We note again that the only requirements on $\mathcal{L}$ are those that guarantee our {\normalfont\textbf{Assumptions}} \hyperref[KAssumptions]{$K$} on the corresponding  heat kernel $K$, that is, the heat kernel giving, e.g., a family of probability measures which are absolutely continuous with respect to the Haar measure on the group, if $X$ is a Lie group.

\begin{example}
Let $G$ be a connected Lie group $\mathfrak{g}$ its Lie algebra. Let $X_1,\ldots,X_d$ be a basis of $\mathfrak{g}$. Let $\mathcal L=-\sum_{i,j} a_{ij}X_i X_j$, where $a=(a_{ij})$ is positive semi-definite. By Hunt's theorem \cite{Hunt}, these operators lead to a semigroup family of probability measures $\nu_t$ corresponding to the operator $e^{-t\mathcal{L}}$. If the family of measures $\nu_t$ is Gaussian (an extra condition to be checked), then the theorem of Wehn (\cite[p. 261]{Wehn}) asserts that the family $\nu_t$ is absolutely continuous with respect to the Haar measure of $G$ if and only if the smallest Lie subalgebra of $\mathfrak{g}$, containing $\mathcal{L}$ in its universal enveloping algebra, coincides with $\mathfrak{g}$ (this is equivalent to H\"ormander's condition). We note here that the absolute continuity then holds simultaneously with respect to left and right Haar measures. This is the case if, for example, $\mathcal{L}$ is H\"ormander's sum of squares.

We note that often the drift terms can be also included into $\mathcal{L}$, see e.g. \cite{Hulanicki} and \cite{Siebert}, or the terms involving the Levi measure \cite{Applebaum, Applebaum-b}. In fact, the Gaussian convolution semigroup $\nu_t$ is non-degenerate if and only if the system $X_j$ is a so-called Siebert system, even for more general locally compact groups, in which case its densities are also strictly positive (see e.g. \cite[p. 314]{BSC}).
\end{example}

Thus, among other settings, our existence and uniqueness results apply to the mild solutions of the semilinear heat equations with concave non-linearity for Laplacians on complete Riemannian manifolds with curvature bounded from below, and for sub-Laplacians on Lie groups.

\end{document}